\documentclass[12pt]{article}
\usepackage{mathtools,fullpage,amsthm,amssymb,enumerate,tikz}
\usetikzlibrary{decorations.pathreplacing}

\usepackage{xcolor}

\usepackage{graphicx}
\usepackage{hyperref}

\newcommand{\arxiv}[1]{\href{http://arxiv.org/abs/#1}{\texttt{arXiv:#1}}}

\newtheorem{theorem}{Theorem}[section]
\newtheorem{lemma}[theorem]{Lemma}

\newtheorem{proposition}[theorem]{Proposition}

\newtheorem{observation}[theorem]{Observation}

\theoremstyle{definition}
\newtheorem{definition}[theorem]{Definition}

\theoremstyle{remark}

\def\marker{\>\hbox{${\vcenter{\vbox{
    \hrule height 0.4pt\hbox{\vrule width 0.4pt height 6pt
    \kern6pt\vrule width 0.4pt}\hrule height 0.4pt}}}$}\>}

\newcommand{\caze}[2]{\textbf{Case {#1}:} \textit{#2}}
\usepackage{xcolor}
\usepackage{dsfont}
\usepackage{marvosym}
\newcommand{\sizeof}[1]{\left\lvert{#1}\right\rvert}
\newcommand{\floor}[1]{\left\lfloor{#1}\right\rfloor}
\newcommand{\ceil}[1]{\left\lceil{#1}\right\rceil}
\newcommand{\chisc}{\chi_{\mathrm{SC}}}
\newcommand{\chisp}{\chi_{\mathrm{SP}}}

\def\C#1{\left|{#1}\right|}
\newcommand{\isc}{\chi_{\mathrm{ISC}}}
\def\spo{\mathring{{\rm s}}}
\def\st{\colon\,}

\def\esub{\subseteq}
\def\nul{\varnothing}
\def\FR{\frac}
\def\FL{\floor}
\def\CL{\ceil}
\def\cost{sum-color cost}
\def\NN{{\mathbb N}}
\def\CH{\binom}
\def\VEC#1#2#3{#1_{#2},\ldots,#1_{#3}}
\def\SE#1#2#3{\sum_{#1=#2}^{#3}}
\def\cR{{\mathcal R}}
\def\cS{{\mathcal S}}
\tikzstyle{vertex}=[inner sep = 0pt, minimum width=6.5pt, fill=black, shape=circle]
\tikzstyle{forest edge}=[line width=2.5pt]
\newcommand{\gpoint}[2]{\node[style=vertex, label=#1:$#2$]}
\newcommand{\apoint}[1]{\gpoint{above}{#1}}
\newcommand{\rpoint}[1]{\gpoint{right}{#1}}

\title{Online Sum-Paintability: Slow-Coloring of Trees}

\author{Gregory J.~Puleo and Douglas B.~West}

\begin{document}
\maketitle

\vspace{-2pc}

\begin{abstract}
The {\it slow-coloring game} is played by Lister and Painter on a
graph $G$.  On each round, Lister marks a nonempty subset $M$ of the
remaining vertices, scoring $\C M$ points.  Painter then gives a
color to a subset of $M$ that is independent in $G$.  The game ends
when all vertices are colored.  Painter's goal is to minimize the
total score; Lister seeks to maximize it.  The score that each
player can guarantee doing no worse than is the \emph{\cost} of $G$,
written $\spo(G)$.  We develop a linear-time algorithm to compute
$\spo(G)$ when $G$ is a tree, enabling us to characterize the $n$-vertex
trees with the largest and smallest values.  Our algorithm also computes on
trees the \emph{interactive sum choice number}, a parameter
recently introduced by Bonamy and Meeks.
\end{abstract}


\section{Introduction}\label{sec:intro}
The {\it slow-coloring game} (introduced in~\cite{MPW}) models the difficulty
of producing a proper coloring of a graph $G$ when it is not known in advance
which vertices are allowed to have which colors.  The players are {\it Lister}
and {\it Painter}.  On the $i$th round, Lister marks a nonempty subset $M$ of
the uncolored vertices, scoring $\C M$ points.  Painter gives color $i$ to a
subset of $M$ that is independent in $G$.  The game ends when all vertices are
colored.  Painter wants to minimize the total score; Lister wants to maximize
it.  The score that each player can guarantee achieving is the \emph{\cost} of
$G$, written $\spo(G)$.

This game is an online version of the ``painting game'', which is an online
version of list coloring.  List coloring generalizes classical graph coloring
by introducing a \textit{list assignment} $L$ that assigns to each vertex $v$ a
set $L(v)$ of available colors.  A graph $G$ is \emph{$L$-colorable} if it has
a proper coloring $\phi$ with $\phi(v)\in L(v)$ for every vertex $v$.  Given
$f\st V(G)\to\NN$, a graph $G$ is \emph{$f$-choosable} if $G$ is $L$-colorable
whenever $\C{L(v)}\ge f(v)$ for all $v$.

Introduced by Vizing~\cite{V2} and by Erd\H{o}s, Rubin, and
Taylor~\cite{ERT}, the \textit{choosability} of a graph $G$ is the
least $k$ such that $G$ is $f$-choosable whenever $f(v)\ge k$ for all
$v\in V(G)$.  Alternatively, we may minimize the sum (or average) of
list sizes.  Introduced by Isaak~\cite{I1,I2} and studied also in
\cite{BBBD,H1,H2,MTW}, the \textit{sum-choosability} of a graph $G$,
denoted $\chisc(G)$, is the minimum of $\sum f(v)$ when $G$ is
$f$-choosable.

In the {\it $f$-painting game}, the color lists are not known in advance.
In round $i$, Lister marks a set $M$ of vertices allowed to receive color $i$;
this can represent the set of vertices having color $i$ in their lists.
Painter chooses an independent subset of $M$ to receive color $i$.  Lister can
design later marked sets based on Painter's choices.  Lister wins if some
vertex is marked more than $f(v)$ times; Painter wins by first coloring all the
vertices.  The graph is \textit{$f$-paintable} if Painter has a winning
strategy.  Introduced by Schauz~\cite{S1} and by Zhu~\cite{Z1},
the \textit{paintability} is the least $k$ such that $G$ is $f$-paintable
whenever $f(v)\ge k$ for all $v\in V(G)$.  Introduced by Carraher et
al.~\cite{CMPW} and studied also in~\cite{MTW}, the \emph{sum-paintability} of
a graph $G$, denoted $\chisp(G)$, is the least value of $\sum f(v)$ for a
function $f$ such that $G$ is $f$-paintable.  Since Lister marks sets in
response to Painter's choices, $\chisp(G)\ge\chisc(G)$ for all $G$.

We view $f$ as allocating tokens to vertices; marking a vertex uses up a token.
Compared to sum-paintability, the slow-coloring game gives some help to Painter
by allowing Painter to postpone allocating tokens to vertices until they are
needed.  The sum-color cost $\spo(G)$ equals the minimum number of tokens
Painter must have available to guarantee producing a coloring.  Since Painter
can always play as if the available tokens are given by a function $f$ such
that $G$ is $f$-paintable, $\spo(G) \leq \chisp(G)$.

Mahoney, Puleo, and West~\cite{MPW} proved various results on $\spo(G)$.
With $\alpha(G)$ denoting the independence number, always
$\FR{\C{V(G)}}{2\alpha(G)}+\FR 12\le\FR{\spo(G)}{\C{V(G)}}\leq 
\max\left\{\frac{\C{V(H)}}{\alpha(H)}\st\!H \subset G\right\}$.  Equality
holds in $\spo(G)\le\chisp(G)$ if and only if all components are complete.
For complete bipartite graphs,
$r+\FR52 s-3+\sqrt{2r-2s}< \spo(K_{r,s})\le r+s+2\sqrt{rs}$ when $r\ge s$,
with $\spo(K_{r,r})\sim 4r$ conjectured.

In this paper, we extend their results on trees.  For $k,r\in\NN$, let
$t_k=\CH{k+1}2$ and $u_r=\max\{k\st t_k \le r\}$.
The numbers $t_k$ are the {\it triangular numbers}.
Note that $u_r = \floor{\frac{-1 + \sqrt{1 + 8r}}{2}}$.

\begin{theorem}[\rm\cite{MPW}]\label{treethm}
For every $n$-vertex tree $T$,
{\rm
\[
n+\sqrt{2n}~\approx~ n+u_{n-1}~ = ~ \spo(K_{1,n-1})
~\le~\spo(T)~\leq~\spo(P_n) ~=~ \floor{{3n}/2}.
\]
}
\end{theorem}

\noindent
Let a {\it stem} in a forest be a vertex having a leaf neighbor and
at most one non-leaf neighbor.  We now state our main result.

\begin{theorem}[Main Theorem]\label{thm:main}
Let $T$ be a forest.  If $T$ has no edges, then $\spo(T)=|V(T)|$.  If $v$ is a
stem in $T$ and $R$ is the set of leaf neighbors of $v$, with $r=|R|$, then
  \[ \spo(T) =
  \begin{cases}
\spo(T-R-v)+r+1+u_r,& \text{if $r+1$ is not a triangular number,}
\\
 \spo(T-R)+r+u_r, &\text{if $r+1$ is a triangular number.}
  \end{cases} \]
\end{theorem}

Since stems are easy to find (the neighbor of a leaf on a longest path is a
stem), this result gives a linear-time algorithm to compute $\spo$ on forests.
We also use it to characterize the extremal trees.

\begin{theorem}\label{charmaxa}
If $T$ is an $n$-vertex forest, then $\spo(T)=\FL{3n/2}$ (the maximum) if and
only if $T$ contains a spanning forest in which every vertex has degree $1$ or
$3$, except for one vertex of degree $0$ or $6$ when $n$ is odd.
If $n\ge4$ and neither $n-1$ nor $n-2$ is a triangular number, then
$\spo(T)=n+u_{n-1}$ (the minimum) if and only if $T$ is a star (in the 
remaining cases, a few additional trees achieve the minimum).
\end{theorem}

Further results have been obtained by Gutowski, Krawczyk, West, Zajac, and
Zhu~\cite{GKWZZ}.  Let $G$ be an $n$-vertex graph.  Always $\spo(G)\le kn$ when
$G$ is $k$-colorable, because Painter can always color at least $\C M/r$
vertices when $M$ is marked.  Hence $\spo(G)\le (1+d)n$ when $G$ is
$d$-degenerate;~\cite{GKWZZ} proves $\spo(G)\le (1+\FR34 d)n$ (the bound must
be at least $(1+\FR12 d)n$, by $\FR n{d+1}K_{d+1}$).  Outerplanar graphs are
$2$-degenerate, but here~\cite{GKWZZ} improves the bound to $\spo(G)\le \FR73n$
(the bound must be at least $2n$, by $\FR n3 K_3$).  When $G$ is planar,
$4$-colorability yields $\spo(G)\le4n$, but~\cite{GKWZZ} improves the bound
to $\spo(G)\le 3.9857n$ (the bound must be at least $\FR52n$, by $\FR n4 K_4$).

A related parameter called the \emph{interactive sum choice number} was
introduced by Bonamy and Meeks~\cite{bonamy-meeks}.  Consider the following
game between two players, whom we call Requester and Supplier. Initially, each
vertex has an empty color list $L(v)$.  In each round of the game, Requester
selects a vertex $v$ and requests a new color for its list; Supplier
chooses a color not already present in $L(v)$ and adds it to the list.
The game continues until the lists $L$ are such that $G$ is
$L$-colorable.  Requester's goal is to minimize the total number of
requests (rounds), while Supplier's goal is to maximize it. The
\emph{interactive sum choice number} of $G$, written $\isc(G)$, is the
common value that both players can guarantee.

The interactive sum choice game is similar to the slow-coloring game, with
Requester analogous to Painter and Supplier analogous to Lister.  The games
differ in who starts each round: in slow-coloring, Painter responds to Lister,
while in interactive sum choice Supplier responds to Requester.  Bonamy and
Meeks posed the problem of computing $\isc(T)$ when $T$ is a forest, giving a
formula for the value on stars, which we state in our notation.
\begin{theorem}[Bonamy--Meeks~\cite{bonamy-meeks}]\label{thm:bonamy-star}
$\isc(K_{1,r})=r+1+u_r$. 
\end{theorem}
The formula in Theorem~\ref{thm:bonamy-star} agrees with the formula 
for $\spo(K_{1,n-1})$ in Theorem~\ref{treethm}, so $\isc(T) = \spo(T)$ when
$T$ is a star.  In Section~\ref{sec:isc-spo}, we prove $\isc(T) = \spo(T)$ for
every forest $T$.  Thus Theorem~\ref{thm:main} also provides an algorithm to
compute $\isc(T)$ on forests.

Bonamy and Meeks proved $\isc(C_n)=3n/2+1$ for the $n$-cycle $C_n$
with $n$ even, but one can obtain $\spo(C_n)=\CL{3n/2}$ using the
value of $\spo(P_n)$; thus $\spo(C_{2k})<\isc(C_2k)$.  Bonamy and Meeks also
noted that $\isc(G)\le\chisc(G)$ for all $G$, since Requester can achieve list
sizes $f$ such that $G$ is $f$-choosable; we ask whether there are graphs with
$\spo(G)>\chisc(G)$.

\section{Basic Observations}
Unlike sum-paintability, \cost\ is given by an easily described (but hard to
compute) recursive formula.  The key point is that prior choices do not affect
Painter's optimal strategy for coloring subsets of marked sets on the remaining
subgraph.

\begin{proposition}[\rm\cite{MPW}]\label{pr:recur}
$\quad\spo(G)=\displaystyle{\max_{\nul\ne M\esub V(G)}}
\left(\C{M} + \min\,\spo(G-I)\right)$,\\ where the minimum is taken over
subsets $I$ of $M$ that are independent in $G$.
\end{proposition}
\begin{proof}
In response to any initial marked set $M$, Painter chooses an independent
subset $I\esub M$ to minimize the cost of the remainder of the game.
\end{proof}

In studying optimal strategies for Lister and Painter, not all legal moves
need be considered.  Let $G[S]$ denote the subgraph of $G$ induced by a set
$S\esub V(G)$.

\begin{observation}[\rm\cite{MPW}]\label{simple}
On any graph, there are optimal strategies for Lister and Painter such that
Lister always marks a set $M$ inducing a connected subgraph $G[M]$, and Painter
always colors a maximal independent subset of $M$.
\end{observation}
\begin{proof}
A move in which Lister marks a disconnected set $M$ can be replaced with
successive moves marking the vertex sets of the components of $G[M]$.
For the second statement, coloring extra vertices at no extra cost cannot hurt
Painter.
\end{proof}

Other easy observations yield useful bounds.  The lower bound below was
observed in~\cite{MPW}.  Let $[A,B]$ denote the set of edges with one endpoint
in $A$ and one endpoint in $B$.

\begin{lemma}\label{lem:cutlemma}\label{cutlemma}
  If $G$ is a graph and $(A,B)$ is a bipartition of $V(G)$, then
\[
\spo(G[A])+\spo(G[B])\le\spo(G)\le\spo(G[A]) + \spo(G[B]) + \sizeof{[A,B]}.
\]
\end{lemma}
\begin{proof}
For the lower bound, Lister can play an optimal strategy on $G[A]$ while
ignoring the rest and then do the same on $G[B]$, achieving the score
$\spo(G[A])+\spo(G[B])$.

For the upper bound, Painter uses optimal strategies on $G[A]$ and $G[B]$.
When doing so requests coloring of both endpoints of an edge in $[A,B]$,
Painter allocates an extra token to each such endpoint $v$ in $B$ and instead
makes the optimal response in $B$ to the marked set obtained by omitting those
vertices from the actual marked set.  Each edge of the cut acts in this way at
most once, because when it does the endpoint in $A$ is colored.
%
\end{proof}

If $\VEC T1k$ are trees such that $|V(T_i)|=n_i$ and $\spo(T_i)=3n_i/2$, then
the disjoint union of $\VEC T1k$ is a forest $F$ with $\spo(F)=3|V(F)|/2$, by
the lower bound in Lemma~\ref{cutlemma}.  Adding edges to turn $F$ into an
$n$-vertex tree $T$ does not reduce the sum-color cost, so $\spo(T)=3n/2$, by
Theorem~\ref{treethm}.  Thus there is a huge variety of trees achieving the
maximum.  Nevertheless, Theorem~\ref{charmaxa} states a simple structural
characterization.

Our algorithm to compute $\spo$ on trees combines the bounds in 
Lemma~\ref{cutlemma} with an understanding of slow-coloring on stars.
By Theorem~\ref{treethm}, $\spo(K_{1,r})=r+1+u_r$.  Therefore, when $r+1$ is
not a triangular number, the statement of Theorem~\ref{thm:main} is
$\spo(T)=\spo(T[A])+\spo(T[B])$, where $B=R\cup\{v\}$ and $A=V(T)-B$.
There is only one edge joining $A$ and $B$, so proving this case only requires
saving $1$ in the upper bound from Lemma~\ref{cutlemma}.

Doing this requires a closer look at optimal play on a star.  The claim we need
includes a computation of $\spo(K_{1,r})$, making our presentation
self-contained.  In \cite{MPW}, $\spo(K_{1,r})$ was computed as a special case
of the join of a complete graph with an independent set.  Our argument for the
special case is somewhat simpler and gives additional information about what
moves are optimal.  Recall that $u_r=\max\{k\st t_k \le r\}$, where
$t_k=\CH{k+1}2$.

\begin{lemma}[\rm\cite{MPW}]\label{uplus}
$u_{r-u_r} = u_r$ when $r+1$ is triangular, and otherwise $u_{r-u_r} = u_r-1$.
\end{lemma}
\begin{proof}
If $u_r=k$, then $t_k\le r<t_{k+1}$.  Also $t_{k+1}-t_k=k+1$.  Thus $r-k=t_k$
if $r+1=t_{k+1}$, yielding $u_{r-u_r}=u_r$.  Otherwise $t_{k-1}<r-k<t_k$, which
yields $u_{r-u_r}=u_r-1$.
\end{proof}

\begin{theorem}\label{thm:split}\label{star}
{\rm$\spo(K_{1,r}) = r+1+u_r$}.  Any optimal first move for Lister marks
the center and $p$ leaves, where $u_r\le p\le u_r+r-t_{u_r}$, or marks only $p$
leaves, where $1\le p\le r-t_{u_r}$.  Painter can respond optimally by coloring
the marked leaves, except that when $r+1$ is triangular and Lister marks the
center and exactly $u_r$ leaves, the only optimal response for Painter is to
color the center.  Finally, if $r+1$ is not triangular and after one round of
optimal play the uncolored subgraph is $K_{1,r'}$, then again $r'+1$ is not
triangular.
\end{theorem}
\begin{proof}
We use induction on $r$, with basis $r=0$ using $u_0=0$.  Suppose $r>0$.

If Lister marks $p$ leaves and not the center, then Painter colors all marked
vertices.  The score is then $p+\spo(K_{1,r-p})$, which by the induction 
hypothesis equals $r+1+u_{r-p}$.  This equals the claimed value (and makes
the Lister move optimal) if $u_{r-p}=u_r$.  By monotonicity of $u$, this holds
if and only $r-p\ge t_{u_r}$, and then $r-p+1$ is again not triangular.

Now suppose that Lister marks the center and $p$ leaves.  Painter responds by
coloring the center or all marked leaves.  If the center is colored, then the
score in the remainder of the game will be exactly $r$.  Otherwise, the game
continues on a star with $r-p$ leaves.  Applying the recurrence of
Proposition~\ref{pr:recur} and the induction hypothesis,
\[
\spo(K_{1,r}) =
\max_p[p+1+\min\{r,\spo(K_{1,r-p})\}]
=r+1+\max_p\min\{p,1+u_{r-p}\}.
\]
Since $u_{r-p}$ is a decreasing function of $p$, with $u_{r-p}>p$ when $p=0$
and $u_{r-p}<p$ when $p=r$, we seek $p$ such that $u_{r-p}=p-1$.

When $r+1$ is not triangular, setting $p=u_r$ yields $u_{r-p}=p-1$, by
Lemma~\ref{uplus}.  Hence $p=u_r$ is optimal for Lister; Painter can color the
center or the leaves.  Smaller $p$ would yield smaller cost.  For $j>0$,
setting $p=u_r+j$ leads to cost $r+1+1+u_{r-p}$, which equals $r+1+u_r$ as long
as $r-p\ge t_{u_r-1}$.  Hence Lister can mark up to $u_r+(r-t_{u_r})$ leaves,
which Painter must color.  When $p=u_r$, Painter may color either the center or
the marked leaves.

In these cases, with Painter coloring leaves, the number $r'$ of leaves
remaining is $r-p$.  We required $p\le u_r+(r-t_{u_r})$ to enforce
$r-p\ge t_{u_r-1}$.  Also, $r<t_{u_{r+1}}-1$ and $p\ge u_r$ yield
$r-p<t_{u_{r+1}}-u_r-1=t_{u_r}-1$.  Thus $t_{u_r-1}\le r-p< t_{u_r}-1$, and
$r'+1$ is not triangular.

When $r+1$ is triangular, setting $p=u_r$ yields $p=u_{r-p}$ (by
Lemma~\ref{uplus}), and Painter must color the center.  However, in this case
also $1+u_{r-u_r-1}=u_{r-u_r}=u_r$.  Thus setting $p$ to be $u_r+1$ again
yields $\min\{p,1+u_{r-p}\}=u_r$, but now Painter must color the leaves to
respond optimally.  As in the previous case, Lister can mark as many as
$u_r+(r-t_{u_r})$ leaves, which in this case equals $2u_r$.  Still Painter
must color the leaves.
\end{proof}

\section{Main Result}\label{sec:main}
When $T$ is a forest of stars, the statement of our main theorem
(Theorem~\ref{thm:main}) reduces to the value given in Theorem~\ref{star}
for stars.  Note that the basis for the main theorem, $\spo(T)=\C{V(T)}$ when
$T$ has no edges, includes the case of the null graph, $\C{V(T)}=0$.

When $T$ has a component that is not a star, the computation in
Theorem~\ref{thm:main} tells us to break off a star having only one nonleaf
neighbor.  The center of such a star is a stem.  Such vertices have sometimes
been called ``penultimate'' vertices, but a stem need not be adjacent to an
endpoint of a longest path.

The cases in proving our main theorem (Theorem~\ref{thm:main}) are based
on Lemma~\ref{uplus}.  Throughout this section, $v$ is a stem in a
forest $T$, the set of leaf neighbors of $v$ is $R$, and $r=|R|\ge1$.
The analysis is easy when $r+1$ is not triangular.

\begin{lemma}\label{nontriang}
If $r+1$ is not triangular, then $\spo(T) = \spo(T-R-v) + \spo(K_{1,r})$.
\end{lemma}
\begin{proof}
With $B=R\cup\{v\}$ and $A=V(T)-B$, we have $T[B] \cong K_{1,r}$, and the lower
bound in Lemma~\ref{lem:cutlemma} yields $\spo(T)\ge\spo(T-R-v)+\spo(K_{1,r})$.
If $N(v) \cap A = \nul$, then the upper bound in Lemma~\ref{lem:cutlemma}
yields the desired equality, so we may assume $v$ has a neighbor in $A$.
\looseness-1

For the upper bound, we use induction on $r$.  When $r=1$, let $w$ and $z$ be
the neighbors of $v$ in $A$ and $B$, respectively.  Let $M$ be the first move
by Lister.  We gain the needed $1$ over the upper bound in
Lemma~\ref{lem:cutlemma} unless $M\cap A$ is an optimal first move in
$T[A]$ and $M\cap B$ is an optimal first move in $T[B]$.  Hence Lister marks
both vertices of $B$, and an optimal response by Painter in $T[B]$ colors one
of them.  Let $S$ be an optimal response for Painter to $M\cap A$.  If $w\in S$,
then Painter colors $S\cup\{z\}$.  If $w\notin S$, then Painter colors
$S\cup\{v\}$.  In either case, the edge $wv$ is gone, and we never allocate an
extra token for it, so optimal play in $A$ and $B$ separately continues,
yielding the desired upper bound.

When $r>1$ and $r+1$ is not triangular, Painter has an optimal response to
optimal play by Lister on $T[A]$ and $T[B]$ separately that colors only leaves
from $T[B]$.  Hence Painter can make this response, leaving $T[A'\cup B']$,
where $A'\esub A$ and $B'\esub B$.  By Lemma~\ref{uplus}, $T[B']$ is a star
with center $v$ and $r'$ leaves, where $r'+1$ is not triangular.  By the
induction hypothesis, Painter has a strategy to complete the game with
additional cost at most $\spo(T[A'])+\spo(K_{1,r'})$.  Since the initial round
was also by optimal play, the total cost is at most $\spo(T[A])+\spo(K_{1,r})$.
\end{proof}

\begin{lemma}\label{lem:penupper}
Always $\spo(T) \leq \spo(T-R)+r+u_r$.
\end{lemma}
\begin{proof}
We use induction on $\C{V(T)}$.
When $v$ is a central vertex of a star component with at least two vertices,
the bound holds with equality, by Theorem~\ref{star}, so the claim holds for
all forests with such components.  Hence we may assume that $v$ has a non-leaf
neighbor and that the inequality holds for stems in all trees with fewer
vertices.

By Lemma~\ref{cutlemma}, $\spo(T-R-v)+r+1+u_r\le\spo(T)\le \spo(T-R-v)+r+u_r+2$.
Lemma~\ref{nontriang} yields $\spo(T)=\spo(T-R-v)+r+1+u_r$ when $r+1$ is not
triangular.  By Lemma~\ref{cutlemma}, $\spo(T-R)\ge\spo(T-R-v)+1$,
and hence $\spo(T)\le \spo(T-R)+r+u_r$ when $\spo(T)=\spo(T-R-v)+r+1+u_r$, such
as when $r+1$ is not triangular.  When $r+1$ is triangular, we give a strategy
for Painter, using the induction hypothesis.

If the desired inequality does not hold, then 
$\spo(T)=\spo(T-R-v)+\spo(K_{1,r})+1$.  Let $M$ be an optimal initial set
marked by Lister; the restrictions of $M$ to $T-R-v$ and $T[R\cup\{v\}]$ must
both be optimal first moves.  Let $p = \sizeof{M \cap R}$, and let $T' = T-R$.

\smallskip
\caze{1}{$v\notin M$.}
Since $r+1=t_{u_r+1}$, we have $r-u_r=t_{u_r}$.  We saw in Theorem~\ref{star}
that marking $p$ leaves and not the center is optimal for Lister on $K_{1,r}$
only when $p\le r-t_{u_r}=u_r$.  Let $X'$ be an optimal reply of Painter to the
move $M-R$ on the subtree $T-R-v$, and let $X = X' \cup (R \cap M)$.  Painter
colors $X$.

In the remaining forest $T-X$, still $v$ is a stem and $v$ has $r-p$ leaf
neighbors.  Since $p\le u_r$, the value $r-p+1$ is not triangular, and
$u_{r-p}=u_r$.  Hence Lemma~\ref{nontriang} applies, and
$\spo(T-X) = \spo(T-X-R-v) + (r-p) + u_{r-p} + 1$. Thus, the optimal total
score is $p+\sizeof{M - R} + \spo(T-X)$, which
equals $\spo(T-X-R-v) + r + u_{r-p} + 1 + \sizeof{M-R}$.  Since $X'$ is an
optimal response to the marked set $M-R$ on $T-R-v$ (and $T-X-R-v=T-X'-R-v$),
we have $\spo(T-X-R-v) + \sizeof{M-R} \leq \spo(T-R-v)$. Since
$\spo(K_{1,r}) = r + u_{r} + 1$ and $u_r = u_{r-p}$, the cost is at most
$\spo(T-R-v) + \spo(K_{1,r})$, which again is at most $\spo(T-R)+r+u_r$.

\smallskip \caze{2}{$v\in M$ and $p > u_r$.}  Let
$M' = M-(R\cup\{v\})$, let $X'$ be an optimal response when Lister
marks $M'$ in $T'$, and let $X = X'\cup (M\cap R)$.  Painter colors
$X$.  The final cost is at most $\sizeof{M}+\spo(T-X)$.  We claim
\[ \spo(T-X) \leq \spo(T'-X) + (r-p) + u_{r-p}. \]
If $p<r$, then $v$ is a stem in $T-X$ and this follows from the induction
hypothesis.  If $p=r$, then $v$ has no leaf neighbors in $T-X$ and is not a
stem, but then $T-X = T'-X$ and $(r-p) + u_{r-p} = 0$.

Also $\sizeof{M} = \sizeof{M'}+p+1$, and $T'-X=T'-X'$, and
$u_{r-p}\le u_r-1$ (since $p>u_r$), so
\begin{align*}
\spo(T)~\le~\sizeof{M}+\spo(T-X)
&~\leq~ \sizeof{M'}+p+1+\spo(T'-X)+(r-p)+u_{r-p} \\
&~\leq~ \sizeof{M'}+\spo(T'-X')+r+u_r ~\leq~ \spo(T')+r+u_r.
\end{align*}

\caze{3}{$v\in M$ and $p \leq u_r$.}
Let $M' = M-R$, and let $X'$ be an optimal response when Lister marks $M'$ in
$T'$.  (If $M'=\nul$, then $X'=\nul$.)  If $v \in X'$, then Painter colors $X'$.
Now $R$ consists of isolated vertices, so $\spo(T)\le\C{M}+\spo(T'-X')+r$.
We compute
\[
\sizeof{M}+\spo(T'-X')+r = p+\sizeof{M'}+\spo(T'-X')+r
\leq \spo(T')+p+r \leq \spo(T')+r+u_r.
\]
If $v \notin X'$, then Painter lets $X=X'\cup (M\cap R)$ and colors $X$.  The
final score is at most $\sizeof{M}+\spo(T-X)$.  Since $v$ remains a stem
in $T-X$, the induction hypothesis yields
$\spo(T-X)\le\spo(T'-X)+(r-p)+u_{r-p}$.  Again $T'-X = T'-X'$, so
\begin{align*}
\sizeof{M}+\spo(T-X) &~\leq~ p+\sizeof{M'}+\spo(T'-X')+(r-p)+u_{r-p} \\
&~\leq~ \spo(T')+r+u_{r-p} ~\leq~ \spo(T') + r + u_r. \qedhere
\end{align*}  
\end{proof}

It remains only to prove equality in Lemma~\ref{lem:penupper}
when $r+1$ is triangular.

\begin{lemma}\label{penlower}
If $r+1$ is triangular, then $\spo(T)\ge \spo(T-R)+r+u_r$.
\end{lemma}
\begin{proof}
Again we use induction on $\sizeof{V(T)}$.  Again the claim holds when $v$ is
the center of a star component, so we may assume that $v$ has a non-leaf
neighbor.  We give a strategy for Lister on the first move and obtain the
desired lower bound for any Painter response.

Since $r+1$ is triangular and $r>0$, we have $r\ge u_r+1$.  Let Lister's 
initial marked set $M$ consist of $v$ and $u_r+1$ vertices from $R$.
By Observation~\ref{simple}, Painter responds by coloring $v$ or $M\cap R$.
Lister plays optimally on the remaining graph.

\caze{1}{Painter colors $v$.}
Let $T'=T-R$; the final score is at least $\sizeof{M} + r + \spo(T'-v)$. 
Since $v$ is a leaf in $T'$, Lemma~\ref{cutlemma} yields
$\spo(T')\le\spo(T'-v)+2$.  Also $\sizeof{M} = u_r+2$, so
\[
\spo(T)\ge \sizeof{M}+r+\spo(T'-v)
= u_r+2+r+\spo(T'-v) \geq u_r + r + \spo(T'),
\]

\caze{2}{Painter colors $M \cap R$.}  Let $T'=T-(M\cap R)$ and $r'=r-(u_r+1)$.
Since $r+1$ is triangular, also $r'+1$ is triangular, and $u_{r'}=u_r-1$.

First suppose $r>2$, so $r'>0$.  Now $v$ is a stem in $T'$, with $r'$
leaf vertices, and $T'-R=T-R$.  The final score is at least
$\sizeof{M}+\spo(T')$.  By the induction hypothesis,
\begin{align*}
\spo(T)&\ge\sizeof{M}+\spo(T') = (u_r + 2) + (u_{r'} + r') + \spo(T'-R) \\
    &= (u_r+2) + (u_r - 1) + (r - u_r - 1) + \spo(T'-R)
    = r + u_r + \spo(T-R).
  \end{align*}
When $r=2$, we have $u_r=1$ and $M=R\cup\{v\}$.  Thus $T'=T-R$.  We compute
\[
\spo(T)\ge\sizeof{M} + \spo(T') = 3 + \spo(T') = r + u_r + \spo(T-R).
\qedhere
\]
\end{proof}

Together, Lemmas~\ref{nontriang}--\ref{penlower} complete the proof of
Theorem~\ref{thm:main}.

\section{Extremal Forests}
Theorem~\ref{thm:main} makes it easy to characterize the $n$-vertex trees whose
sum-color cost equals the upper bound $\FL{3n/2}$ proved in~\cite{MPW}.  The
characterization includes an alternative proof of that bound.  We will do the
same for the trees achieving the lower bound $n+u_{n-1}$.

The characterization of the upper bound when $n$ is even is simple and 
elegant, but for odd $n$ some annoying flexibility creeps in. 

\begin{theorem}\label{charmax}
If $T$ is an $n$-vertex forest, then $\spo(T)\le \FL{3n/2}$.  Furthermore,
equality holds if and only if $T$ contains a spanning forest in which every
vertex has degree $1$ or $3$, except that when $n$ is odd there is also one
vertex with degree $0$ or $6$.
\end{theorem}
\begin{proof}
Say that a forest is {\it tight} if it contains a spanning forest as described
in the statement.  Call such a spanning forest a {\it witness}.  Note that
a witness has a vertex of even degree if and only if $n$ is odd.  Say that
a tree is even when the number of vertices is even; otherwise it is odd.

We use induction on $n$.  The claims hold by inspection when $T$ has no edges,
including when $T$ has no vertices.  Hence we may assume that $T$ has a stem
$v$ (possibly the center of a star component).  Let $R$ be the set of leaf
neighbors of $v$, with $r=\sizeof{R}$, and let $T'=T-R-v$.

We first prove the upper bound, using Theorem~\ref{thm:main}.  Note that
$r+1+u_r\le \FL{3(r+1)/2}$, with equality only for $r\in\{1,2,3,4,6\}$.
If $r+1$ is not triangular, then the induction hypothesis yields
$\spo(T)=\spo(T')+r+1+u_r\le \FL{3(n-r-1)/2}+\FL{3(r+1)/2}\le \FL{3n/2}$.
If $r+1$ is triangular, then $r+u_r\le\FL{3r/2}$, with equality (among
triangular $r+1$) only for $r\in\{2,5\}$.  Hence when $r+1$ is triangular,
the induction hypothesis yields
$\spo(T)=\spo(T-R)+r+u_r\le \FL{3(n-r)/2}+\FL{3r/2}\le \FL{3n/2}$.

This completes the proof of the upper bound.  It remains to prove the 
characterization of equality (again by induction, with the basis when
$T$ has no edges).  Some cases in the proof are shown in Figure~\ref{fig:cases}.
The upper bound implies that when proving sufficiency, we
only need to prove $\spo(T)\ge\FL{3n/2}$ when $T$ is tight.

\begin{figure}[hbt]
  \centering
  \begin{tikzpicture}
    \begin{scope}[xshift=-4cm]
      \apoint{v} (v) at (0cm, 0cm) {};
      \apoint{} (w) at (-1cm, 0cm) {};
      \apoint{} (z1) at (1cm, 1cm) {};
      \apoint{} (z2) at (1cm, .33cm) {};
      \apoint{} (z3) at (1cm, -.33cm) {};
      \apoint{} (z4) at (1cm, -1cm) {};      
      \draw (w) -- (v);
      \foreach \i in {1,...,3} { \draw[forest edge] (v) -- (z\i); }
      \draw (v) -- (z4);
      \draw (-1.5cm, 0cm) circle (1cm);
      \node at (-1.75cm, 0cm) {$T'$};
      \node at (-.75cm, -2cm) {(a)};
    \end{scope}
    \begin{scope}[xshift=1cm]
      \apoint{v} (v) at (0cm, 0cm) {};
      \apoint{} (w) at (-.8cm, 0cm) {};
      \apoint{} (z1) at (1cm, .6cm) {};
      \apoint{} (z2) at (1cm, -.6cm) {};
      \draw[forest edge] (w) -- (v);
      \draw[forest edge] (v) -- (z1);
      \draw[forest edge] (v) -- (z2);
      \draw (-.7cm, 0cm) circle (1cm);
      \node at (-.7cm, -.5cm) {$T-R$};
      \node at (-.75cm, -2cm) {(b)};      
    \end{scope}
    \begin{scope}[xshift=6cm]
      \apoint{v} (v) at (0cm, 0cm) {};
      \apoint{} (w) at (-.8cm, 0cm) {};
      \apoint{} (z1) at (1cm, .6cm) {};
      \apoint{} (z2) at (1cm, -.6cm) {};
      \draw (w) -- (v);
      \draw[forest edge] (v) -- (z1);
      \draw (v) -- (z2);
      \draw (-.7cm, 0cm) circle (1cm);
      \node at (-.7cm, -.5cm) {$T-R$};
      \node at (-.75cm, -2cm) {(c)};      
    \end{scope}
  \end{tikzpicture}
  \caption{Selected cases for Theorem~\ref{charmax}.}
  \label{fig:cases}  
\end{figure}
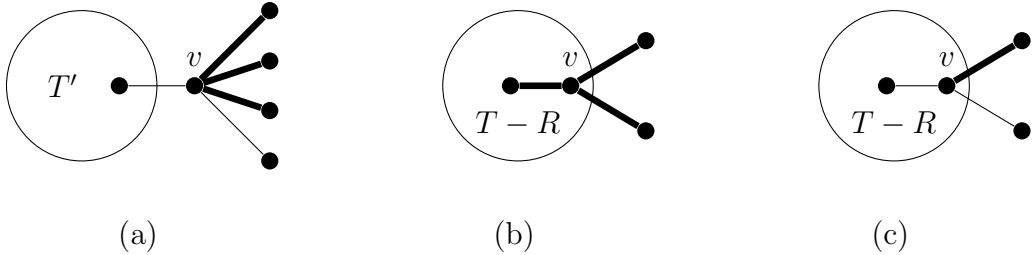

\medskip
{\bf Necessity.}
We assume $\spo(T)=\FL{3n/2}$.  The upper bound computation yields $r\le6$.

If $r+1$ is not triangular, then $r\in\{1,3,4,6\}$.  Equality
in the bound requires $\spo(T')=\FL{3\C{V(T')}/2}$, so by the
induction hypothesis $T'$ is tight.  Thus $T'$ has a witness.  If
$r\in\{1,3\}$, then adding the edges from $v$ to $R$ yields a
witness for $T$.  If $r\in\{4,6\}$, then $\FL{3(r+1)/2}<3(r+1)/2$, so
equality in $\spo(T)\le\FL{3n/2}$ requires $T'$ to be even.  Hence a
witness for $T'$ has no vertex of even degree, and we obtain a witness
for $T$ by adding three edges from $v$ to $R$ if $r=4$ and adding all
six edges if $r=6$. Figure~\ref{fig:cases}(a) shows the case $r=4$.

Now suppose $r+1$ is triangular, so $r\in\{2,5\}$.
Equality requires $\spo(T-R)=\FL{3\C{V(T')}/2}$, and hence $T-R$ is tight.
When $r=5$, we have $\FL{3r/2}=(3r-1)/2$, so $\spo(T)=\FL{3n/2}$ requires $T-R$
to be even.  Hence $v$ has degree $1$ in a witness for $T-R$, and adding all
the edges from $v$ to $R$ completes a witness for $T$, with $v$ of degree $6$.

When $r=2$, again $T-R$ must be tight.  If a witness has degree $1$ at
$v$, as in Figure~\ref{fig:cases}(b), then add both edges from $v$ to $R$
to make a witness for $T$.  If instead $v$ is isolated, as in
Figure~\ref{fig:cases}(c), then add one edge from $v$ to $R$ and leave
the other vertex of $R$ isolated.

\medskip
{\bf Sufficiency.}
We assume that $T$ is tight.

If a witness $W$ for $T$ has some vertex $z\in R$ as an isolated vertex, then
$W-z$ is a witness for $T-z$, and $T-z$ is even.  By the induction hypothesis,
$\spo(T-z)=3(n-1)/2$.  Since Lister can play $T-z$ and $z$ separately,
$\spo(T)\ge\FL{3n/2}$, and hence equality holds.

Hence we may assume that all vertices of $R$ have degree $1$ in $W$.
This requires $d_W(v)\in\{1,3,6\}$.  If $d_W(v)=r$, then $r+1$ is not
triangular.  Also $T'$ is tight (and is even if $d_W(v)=6$).  By the induction
hypothesis, equality holds for $T'$ and in the computation for $T$.

Hence we may assume that $v$ has a non-leaf neighbor in $W$, so
$r=d_W(v)-1$, which eliminates the case $r=1$.  Note that $r+1\in\{3,6\}$,
so $r+1$ is triangular, and $\spo(T)=\spo(T-R)+\FL{3r/2}$.
For $r=2$, we have $d_W(v)=3$.  Hence $W-R$ is a witness for $T-R$.
When $r=5$ and $d_W(v)=6$, tightness of $T$ requires $T$ to be odd.
Deleting the five vertices of $R$ from $W$ leaves a witness for $T-R$.
In both cases, the induction hypothesis gives equality hold in the bound for
$T-R$ and hence in the bound for $T$.
\end{proof}

Our final task is to characterize the $n$-vertex trees having the least cost.
When neither $n-1$ nor $n-2$ is triangular, the star $K_{1,n-1}$ is the unique
minimizing tree, but in the remaining cases a few other trees may also have
this cost.  For example, from Lemma~\ref{nontriang} and the formula 
$\spo(K_{1,r})=r+1+u_r$, it is easy to check that when $n-1$ or $n-2$ is
triangular the tree obtained by subdividing one edge of $K_{1,n-2}$ has
the same cost as $K_{1,n-1}$.

Our proof of the characterization provides an alternative proof of
the result in~\cite{MPW} that the minimum is $n+u_{n-1}$.
We begin with a numerical lemma, which suggests that the tree obtained by
subdividing one edge of $K_{1,n-2}$ is the best candidate to match $K_{1,n-1}$.




\begin{lemma}\label{concave}
If $m,r\in\NN$ satisfy $2<r\le m/2$, then $u_r+u_{m-r}\ge u_1+u_{m-1}$.  If
also no number in $\{m-4,m-3,m-2,m-1\}$ is triangular, then
$u_r + u_{m-r} \geq 1 + u_1 + u_{m-1}$.
\end{lemma}
\begin{proof}
Let $a_r=u_r+u_{m-r}$.  To prove the first statement, it suffices to
show for $2<r\le m/2$ that there exists $q$ with $1\le q<r$ and
$q\ne2$ such that $a_q\le a_r$.

First consider $r\in\{3,4,5\}$ and $m\ge 2r$.  We have $a_1=1+u_{m-1}$ and
$a_r=2+u_{m-r}$.  When $r\in\{3,4,5\}$ and $m\ge 2r$, there cannot be two
triangular numbers in the interval $[m-r+1,m-1]$, so
$u_{m-r}\ge u_{m-1}-1$.  Thus $a_1\le a_r$, so setting $q=1$ suffices.

For $r\ge6$, first compare $a_r$ with $a_{r-1}$:
\begin{equation*}
  a_{r}-a_{r-1}=\begin{cases}
    1& \text{ if $r$ is triangular and $m-r+1$ is not,}\\
    -1& \text{ if $m-r+1$ is triangular and $r$ is not,}\\
    0& \text{ otherwise.}
  \end{cases}
\end{equation*}
If $m-r+1$ is not triangular or $r$ is triangular, then setting $q=r-1$
suffices.  Otherwise, $m-r+1$ is triangular and $r$ is not.  Let
$q=\CH{u_r+1}2-1$; note that $u_q=u_r-1$.  Also, the difference between
$m-r+1$ and the next higher triangular number is greater than $r-q$; hence
$u_{m-q}=u_{m-r+1}=u_{m-r}+1$.  Thus $a_q=a_r$.  Also, since $r>6$, we have
$q\ge5$, so we eventually reduce to $r=5$.  This proves the first statement.

Now suppose that none of $\{m-4,m-3,m-2,m-1\}$ is triangular.  In this case, for
$r \in \{3,4,5\}$, there is no triangular number in the interval $[m-r+1,m-1]$;
hence $u_{m-r} = u_{m-1}$ and $a_r \geq a_1+1$.  For $r \geq 6$, the argument
above yields $q$ with $r>q\ge 5$ such that $a_q \leq a_r$, which now implies
the stronger inequality $a_r \geq a_1+1$.
\end{proof}

\begin{theorem}\label{min-nontri}
If $T$ is an $n$-vertex tree, then $\spo(T)\ge n+u_{n-1}$, with equality
for $T=K_{1,n-1}$.  Furthermore, $K_{1,n-1}$ is the unique minimizing tree
when neither $n-1$ nor $n-2$ is triangular.
\end{theorem}
\begin{proof}
We proved $\spo(K_{1,n-1})=n+u_{n-1}$ in Theorem~\ref{star}.  To complete the
proof, we show that $\spo(T)\ge\spo(K_{1,n-1})$ for every non-star tree, with
strict inequality when neither $n-1$ nor $n-2$ is triangular.
We use induction on $n$.  Every tree with at most three vertices is a star,
so we may assume $n\ge4$.

Let $v$ be a stem in $T$, and let $R$ be the set of leaf neighbors of
$v$, with $r = \sizeof{R}$.  Since we may let $v$ be the penultimate vertex on
either end of a longest path, we may assume $r\le (n-2)/2$.  Let $m=n-2$.

We claim $\spo(T)\ge n+1+u_{n-3}$.  If so, then since $n\ge4$ implies
$u_{n-3}\ge u_{n-1}-1$, we have $\spo(T)\ge \spo(K_{1,n-1})$.
If also neither $n-1$ nor $n-2$ is triangular, then $u_{n-3}=u_{n-1}$, and
$\spo(T)>\spo(K_{1,n-1})$.  It remains to prove the claim.

If $r=2$, then since $r+1$ is triangular, Theorem~\ref{thm:main} and the
induction hypothesis yield
\[
\spo(T)=\spo(T-R)+r+u_r = \spo(T-R)+3 \geq (n-2+u_{n-3})+3 = n+u_{n-3}+1.
\]

If $r\ne2$, then we use Theorem~\ref{thm:main}, the lower bound in
Lemma~\ref{cutlemma} (yielding $\spo(T-R)\ge\spo(T-R-v)+1$), the induction
hypothesis, and Lemma~\ref{concave} for $r\ne2$ (with $m=n-2$) to compute
\begin{align*}
\spo(T)\ge\spo(T-R-v)+r+1+u_r&\ge
n-r-1+u_{n-r-2}+r+1+u_r\\
&= n+u_r+u_{n-r-2}\ge n+1+u_{n-3}.
\end{align*}
This completes the proof.
\end{proof}

The remaining case is somewhat technical; additional minimizing trees arise.
They are close to being stars.  For $a,b\ge1$, the {\it double-star} $S_{a,b}$
is the tree with $a+b+2$ vertices having two non-leaf vertices, one with $a$
leaf neighbors and the other with $b$ leaf neighbors (see
Figure~\ref{fig:lowertrees}).  The {\it subdivided double-star} $S'_{a,b}$ is
obtained from $S_{a,b}$ by subdividing the central edge to add one vertex.
Note that the tree obtained from $K_{1,n-2}$ by sudividing one edge is
$S_{1,n-3}$.

\begin{figure}[h]
  \centering
  \begin{tikzpicture}
    \begin{scope}[xshift=-5cm]
      \apoint{} (v) at (0cm, 0cm) {};
      \apoint{} (w1) at (-2cm, 1cm) {};
      \apoint{} (w2) at (-1cm, 1cm) {};
      \node at (0cm, 1cm) {$\cdots$};
      \apoint{} (w3) at (1cm, 1cm) {};
      \apoint{} (w4) at (2cm, 1cm) {};
      \rpoint{} (y) at (0cm, -1cm) {};
      \apoint{} (z1) at (-2cm, -2cm) {};      
      \apoint{} (z2) at (-1cm, -2cm) {};      
      \node at (0cm, -2cm) {$\cdots$};
      \apoint{} (z3) at (1cm, -2cm) {};      
      \apoint{} (z4) at (2cm, -2cm) {};      
      \draw [decorate, decoration={brace, amplitude=5pt}] (-2.2cm, 1.1cm) -- (2.2cm, 1.1cm) node[above, yshift=5pt, pos=.5] {$a$ leaves};
      \draw [decorate, decoration={brace, amplitude=5pt}] (2.2cm, -2.1cm) -- (-2.2cm, -2.1cm) node[below, yshift=-5pt, pos=.5] {$b$ leaves};
      \foreach \i in {1,...,4} { \draw (v) -- (w\i); }
      \foreach \i in {1,...,4} { \draw (y) -- (z\i); }
      \draw (v) -- (y);
      \node at (0cm, -3cm) {$S_{a,b}$};
    \end{scope}
    \begin{scope}[xshift=0cm]
      \apoint{} (v) at (0cm, 0cm) {};
      \apoint{} (w1) at (-2cm, 1cm) {};
      \apoint{} (w2) at (-1cm, 1cm) {};
      \node at (0cm, 1cm) {$\cdots$};
      \apoint{} (w3) at (1cm, 1cm) {};
      \apoint{} (w4) at (2cm, 1cm) {};
      \rpoint{} (y) at (0cm, -1cm) {};
      \rpoint{} (x) at (.5cm, -.5cm) {};
      \apoint{} (z1) at (-2cm, -2cm) {};      
      \apoint{} (z2) at (-1cm, -2cm) {};      
      \node at (0cm, -2cm) {$\cdots$};
      \apoint{} (z3) at (1cm, -2cm) {};      
      \apoint{} (z4) at (2cm, -2cm) {};      
      \draw [decorate, decoration={brace, amplitude=5pt}] (-2.2cm, 1.1cm) -- (2.2cm, 1.1cm) node[above, yshift=5pt, pos=.5] {$a$ leaves};
      \draw [decorate, decoration={brace, amplitude=5pt}] (2.2cm, -2.1cm) -- (-2.2cm, -2.1cm) node[below, yshift=-5pt, pos=.5] {$b$ leaves};
      \foreach \i in {1,...,4} { \draw (v) -- (w\i); }
      \foreach \i in {1,...,4} { \draw (y) -- (z\i); }
      \draw (v) -- (x) -- (y);
      \node at (0cm, -3cm) {$S'_{a,b}$};
    \end{scope}
  \end{tikzpicture}
  \caption{Non-star trees achieving the lower bound in Theorem~\ref{thm:classify}.}
  \label{fig:lowertrees}
\end{figure}
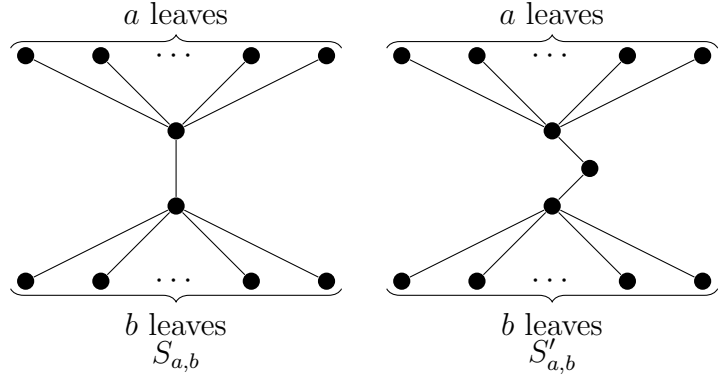

\begin{theorem}\label{thm:classify}
When $n-1$ or $n-2$ is triangular, the $n$-vertex trees $T$ minimizing
$\spo(T)$ are $\{K_{1,n-1},S_{1,n-3},S_{2,n-4},S'_{1,n-4}\}$, except that
for $n=7$ all trees are included, and $S_{4,5}$ and $S'_{4,4}$ are included
when $n=11$.
\end{theorem}
\begin{proof}
For $n \in \{4,5,7\}$, the upper and lower bounds on $\spo(T)$ are equal, so
all $n$-vertex trees achieve the lower bound.  When $n\in\{4,5\}$, all 
$n$-vertex trees are in the listed set.  Thus, we may use induction on $n$
with basis $n\le 7$.  Suppose $n \geq 8$.

Let $T$ be an $n$-vertex tree achieving $\spo(T) = n+u_{n-1}$.  Let $v$ be a
stem having the fewest leaf neighbors, $R$ the set of leaves adjacent to
$v$, and $r = \sizeof{R}$.  We may assume that $T$ is not a star,
which implies $r\le \frac{n-2}{2}$.

Since $n-1$ or $n-2$ is triangular, $u_{n-3}=u_{n-1}-1$.  Furthermore,
since $n\ge8$, the numbers $n-3$ and $n-4$ are not triangular, which means
$u_{n-4}=u_{n-5}=u_{n-1}-1$.

If $r=1$, then $r+1$ is not triangular, and Theorem~\ref{thm:main} yields 
$\spo(T)=\spo(T-R-v)+3$.  Also $T-R-v$ has $n-2$ vertices, so
$\spo(T)\ge (n-2)+u_{n-3}+3 =n+u_{n-1}$.  Since neither $n-3$ nor $n-4$ is
triangular, by the induction hypothesis equality holds if and only if
$T-R-v=K_{1,n-3}$.  Depending on whether $v$ is adjacent to the center or a
leaf of $T-R-v$, we have equality if and only if $T$ is $S_{1,n-3}$ or
$S'_{1,n-4}$.

If $r=2$, then $r+1$ is triangular, and Theorem~\ref{thm:main} yields 
$\spo(T)=\spo(T-R)+3$.  Now $T-R$ has $n-2$ vertices, so the same argument
as above yields $T-R=K_{1,n-3}$ when equality holds.  Our vertex $v$ lies in
$K_{1,n-3}$, and $T$ is obtained by adding two pendant edges at $v$.  If $v$
is the center of $K_{1,n-3}$, then $T$ is a star; hence $v$ is a leaf of
$K_{1,n-3}$ and $T = S_{2,n-4}$.

Hence we may assume $r\ge3$.  Let $\epsilon=1$ if $\{n-3,n-4,n-5,n-6\}$
contains no triangular number, and otherwise $\epsilon=0$.  As in
Theorem~\ref{min-nontri}, we use the lower bound in Lemma~\ref{cutlemma}, the
lower bound for trees with $n-r-1$ vertices, Lemma~\ref{concave} (with $m=n-2$),
and $u_{n-3}=u_{n-1}-1$ to compute
\begin{equation}
  \label{eq:cuttree}
  \spo(T) \geq \spo(T-R-v) + r+1+u_r = n + u_{n-r-2}+u_{r}
\geq n+\epsilon+u_{n-3}+u_1 = n +\epsilon+ u_{n-1}.
\end{equation}
If $\epsilon=1$, then $\spo(T)>\spo(K_{1,n-1})$.  Since $n-1$ or $n-2$ is
triangular, when $n\ge17$ no triangular number lies in $\{n-3,n-4,n-5,n-6\}$.
The only cases remaining are $n\in\{8,11,12,16\}$, which yield
$\epsilon=0$.  Now equality must hold throughout in (\ref{eq:cuttree}),
which requires $\spo(T-R-v) = (n-r-1) + u_{n-r-2}$.  In other words, $T-R-v$ is
a minimizing tree on $n-r-1$ vertices, and $\spo(T)=n+u_{n-r-2}+u_r$.
Equaling the minimum requires $u_{n-r-2}+u_r=u_{n-1}$.

If $n=8$, then $r \leq \frac{n-2}{2}$ forces $r=3$. (Recall that we
are assuming $r \geq 3$ here.) Since $u_3+u_3=4>3=u_7$, the only
minimizing trees are $\{K_{1,7},S_{1,5},S_{2,4},S'_{1,4}\}$.

If $n=11$, then $r \leq \frac{n-2}{2}$ forces $r\le4$.
Since $u_6+u_3=5>4=u_{10}$, the case $r=3$ is eliminated.
If $r=4$, then $\spo(T)=11+u_5+u_4=15$, and $T-R-v$ is a minimizing tree on six
vertices.  Thus $T-R-v=K_{1,5}$, which yields $T\in\{S_{4,5},S'_{4,4}\}$,
depending on whether $v$ is adjacent to the center or a leaf of $K_{1,5}$.

If $n=12$, then $r \leq \frac{n-2}{2}$ forces $r\le5$.
Note that $u_7+u_3=u_6+u_4=5>4=u_{11}$, so we may assume $r=5$.
Since $r+1$ is triangular, Theorem~\ref{thm:main} eliminates this case via
\[ \spo(T) = \spo(T-R) + 5 + u_5 \geq 7 + u_{6} +5 + u_5 = 17. \]

If $n=16$, then $r \leq \frac{n-2}{2}$ forces $r\le7$.
Since $u_r+u_{14-r}=6>5=u_{15}$ for $r \in \{3,4,6,7\}$, such values of $r$
cannot occur in a minimizing tree.  When $r=5$ we again apply
Theorem~\ref{thm:main} to eliminate this case via
\[\spo(T)=\spo(T-R)+5+u_5 \geq 11+u_{10}+5+u_5 = 22>21=16+u_{15}. \]
This completes the proof.
\end{proof}

Theorems~\ref{min-nontri} and \ref{thm:classify} complete the characterization
of the minimizing trees.

\section{The Interactive Sum Choice Number of Forests}\label{sec:isc-spo}

Recall that in the game introduced by Bonamy and Meeks~\cite{bonamy-meeks}, in
each round Requester specifies a vertex $v$ of the graph $G$ and Supplier adds
a color to the list $L(v)$, with the game ending when $G$ is $L$-colorable. 
The length of the game under optimal play is $\isc(G)$.

We prove $\spo(T) = \isc(T)$ for each forest $T$ by showing that $\isc$
satisfies the same recurrence as $\spo$ on trees, using lemmas like those in
Section~\ref{sec:main}.  That is, we prove the following:

\begin{theorem}\label{thm:iscmain}
Let $T$ be a forest.  If $T$ has no edges, then $\isc(T)=|V(T)|$.  If $v$ is a
stem in $T$ and $R$ is the set of leaf neighbors of $v$, with $r=|R|$, then

\vspace{-1pc}
  \[ \isc(T) =
  \begin{cases}
\isc(T-R-v)+r+1+u_r,& \text{if $r+1$ is not a triangular number,}
\\
 \isc(T-R)+r+u_r, &\text{if $r+1$ is a triangular number.}
  \end{cases} \]
\end{theorem}

Interactive sum coloring satisfies bounds like Lemma~\ref{lem:cutlemma} for
slow coloring when a graph is broken into subgraphs by a vertex bipartition.
The lemma is a special case of one by Bonamy and Meeks~\cite{bonamy-meeks},
which we rephrase slightly.  We include a proof for completeness.

\begin{lemma}[Bonamy--Meeks~\cite{bonamy-meeks}]\label{lem:cut}
If $G$ is a graph and $(A,B)$ is a partition of $V(G)$, then

\vspace{-1pc}
\[\isc(G[A])+\isc(G[B])\le \isc(G)\le\isc(G[A])+\isc(G[B])+\C{[A,B]}.\]
\end{lemma}
\begin{proof}
As in Lemma~\ref{lem:cutlemma}, the lower bound holds because Supplier can
respond in the games on $G[A]$ and $G[B]$ separately.  For the upper bound,
Requester can play optimal strategies first on $G[A]$ (producing a proper
coloring $\phi$ from the resulting lists) and then on $G[B]$, with an extra
request made at the endpoint $y$ in $B$ of an edge $xy$ in $[A,B]$ whenever
Supplier provides the color $\phi(x)$ at $y$.  Supplier can only provide that
color once for each such edge, and Requester can put aside that response and
continue on $G[B]$ as if it never happened.
\end{proof}

Recall from Theorem~\ref{thm:bonamy-star} that $\isc(K_{1,r})=r+1+u_r$.  Also
$\isc(G)=\C{V(G)}$ when $\C{V(G)}\le1$.  When $T$ is a star, both cases
in Theorem~\ref{thm:iscmain} give the known formula $\isc(T)=r+1+u_r$.  Also
the parameter is additive over components.  Thus, we may assume that $T$ has a
component that is not a star and thus has a stem with a non-leaf neighbor.

Henceforth $v$ is a stem with non-leaf neighbor $w$ in a non-star component of
a forest $T$, the set of leaf neighbors of $v$ is $R$, and $r =\C R\ge1$.  Let
$R'=R\cup\{v\}$ and $T' = T-R'=T-R-\{v\}$.
Lemma~\ref{lem:cut} and Theorem~\ref{thm:bonamy-star} yield
\begin{equation} \label{iscbd}
\isc(T') + r + u_r + 1 \leq \isc(T) \leq \isc(T') + r + u_r + 2.
\end{equation}

For Theorem~\ref{thm:iscmain}, in the case where $r+1$ is not triangular we
need a strategy for Requester that improves the upper bound in \eqref{iscbd}
by $1$.

\begin{definition}\label{cplay}
Given a stem $v$ with a color $c$ in its current list,
{\it freeing $c$ at $v$} means requesting an additional color at each
leaf neighbor of $v$ whose current list is precisely $\{c\}$.
\end{definition}

Note that freeing $c$ at $v$ may make $c$ available for use at $v$ in a proper
coloring chosen from the lists; it also ensures that each leaf neighbor of $v$
with $c$ in its list can be colored.

\begin{lemma}\label{lem:nontriang}
If $r+1$ is not triangular, then $\isc(T) = \isc(T')+r+u_r+1$.
\end{lemma}
\begin{proof}
Due to \eqref{iscbd}, it suffices to prove
$\isc(T)\le \isc(T')+r+u_r+1$.  We provide a strategy for Requester that is
a slight modification of the strategy used on stars in \cite{bonamy-meeks}.

Requester first requests an initial color at each vertex of $R'$.  For each
$x\in R'$, let $\alpha(x)$ be the first color supplied by Supplier at $x$.
Requester's subsequent strategy is in three phases.  Phases 1 and 3 involve
making requests in $R'$; Phase 2 involves playing optimally on $T'$.  During
the game, let $i$ denote the number of requests that have been made so far at
$v$, let $c_i$ be the color supplied in response to the $i$th request at $v$,
and let $S_i=\{z\in R\st\alpha(z)=c_i\}$.  After the initial colors are
supplied at $R'$, we have $i=1$ and $c_1=\alpha(v)$.  

\medskip
{\it Phase 1}.  Phase 1 requests colors at $v$.  As long as
$\C{S_i}>u_r-i+1$, Requester obtains color $c_{i+1}$ at $v$ and increments $i$.
When $\C{S_i}\le u_r-i+1$, no request is made, Phase 2 begins, and we set
$i^*=i$.  There is no request in Phase 1 if $\C{S_1}\le u_r$.

{\it Phase 2}.  Requester plays on $T'$.  There are two cases, depending
on whether $\C{S_i}$ equals $u_r-i+1$ or is smaller.

If $\C{S_{i^*}}\le u_r-{i^*}$, then Requester first frees $c_{i^*}$ at $v$ and
then plays an optimal request sequence on $T'$.  However, if Supplier adds
$c_{i^*}$ to the list at $w$ in response to a request there, then Requester
ignores that move and immediately makes an extra request at $w$.  The copy of
$c_{i^*}$ becomes an extra unused color in $L(w)$.  In the game played
optimally by Requester on $T'$, the list at $w$ is considered not to contain
$c_{i^*}$.  An $L$-coloring $\phi$ of $T'$ then exists without using $c_{i^*}$
at $w$.  With $c_{i^*}$ available at $v$, the proper coloring $\phi$ extends to
$R'$.

If $\C{S_{i^*}}=u_r-{i^*}+1$, then Requester next plays an optimal request
sequence on $T'$.  From the colors supplied, Requester can choose a proper
coloring $\phi$ of $T'$.  Let $c'=\phi(w)$.  If $c'\ne c_{i^*}$, then Requester
frees $c_{i^*}$ at $v$ (by adding colors at the leaf neighbors of $v$ in $R$),
which makes it possible to use $c_{i^*}$ at $v$ and extend $\phi$ to all of $T$.
If $c'=c_{i^*}$, then the process moves to Phase 3.

{\it Phase 3}.  Here, in the situation $\C{S_{i^*}}=u_r-i^*+1$ and $c'=c_{i^*}$,
Requester increments $i$ (to $i^*+1$) and requests $c_{i^*+1}$ at $v$.
As long as $\C{S_i}>u_r-i+1$, Requester increments $i$ and requests a
color at $v$.  At the point when $\C{S_i}\le u_r-i+1$, Requester frees
$c_i$ at $v$.  Since $c'=\phi(w)=c_{i^*}$, the new color at $v$ is different
from the color $c'$ at $w$ under $\phi$, so it can be used at $v$ and the
coloring extends to all of $T$.

\medskip


Any time a request is made in Phase 1 or Phase 3 to pick another color at
$v$ after $c_i$, we have $\C{S_i}\ge u_r-i+2$.  If we reach Phase 3, then
we also have $\C{S_{i^*}}= u_r-i^*+1$.  The sets of the form $S_i$ are pairwise
disjoint subsets of $R$.  There are at most $u_r$ requests in Phases 1 and 3,
since making $u_r+1$ such requests requires
$
r=\C R\ge\SE i1{u_r+1}\C{S_i}\ge \CH{u_r+2}2-1.
$
Since $r+1$ is not triangular, we also have $r\ge\CH{u_r+2}2$,
but by definition $u_r=\max\{k\st \CH{k+1}2\le r\}$.
Hence the strategy terminates, with at most $1+u_r$ requests at $v$.

%
%
  
We have observed that this strategy produces lists from which a proper coloring
can be chosen on $T$, given an optimal strategy on $T'$.  We claim also that it
uses at most $r + u_r + 1 + \isc(T')$ requests altogether.


Suppose first that $\C{S_{i^*}}\le u_r-i^*$.  In this case, we claim that the
total number of requests is at most
  $ r + i^* + \sizeof{S_{i^*}} + (\isc(T') + 1) $.
After the initial $r+1$ requests on $R'$ come $i^*-1$ additional requests at
$v$ in Phase 1.  Freeing $c_{i^*}$ at $v$ in Phase 2 does not make a request at
$v$ but makes $\C{S_{i^*}}$ requests in $R$.  There are then at most
$\isc(T')+1$ requests on $T'$ in Phase 2 (there may be an extra request at
$w$).  Since $\sizeof{S_{i^*}} \leq u_r - i^*$, the total number of requests is
at most $r + u_r + 1 + \isc(T')$, as desired.

If $\C{S_{i^*}}= u_r-i^*+1$, then let $c_k$ be the last color supplied at $v$.
We claim that the total number of requests is at most
$ r + k + \sizeof{S_k} + \isc(T') $.
Initially there are $r+1$ requests on $R'$, and eventually there are $k-1$
additional requests at $v$.  In Phase 2 there are at most $\isc(T')$ requests
on $T'$.  If $c'\ne c_{i^*}$, then $i^* = k$, so freeing $c_{i^*}$ uses
$\C{S_k}$ additional requests on $R$. On the other hand, if $c'=c_{i^*}$, then
later freeing $c_k$ at $v$ incurs $\C{S_k}$ additional requests on $R$; again
the claimed bound holds.  In both cases $\sizeof{S_k} \leq u_r - k +1$,
yielding at most $r + 1 + u_r + \isc(T')$ total requests.
\end{proof}

When $r+1$ is triangular, we want to compute $\isc(T)$ in terms of the subtree
obtained by deleting only $R$, not $R\cup\{v\}$.  That is,
$\isc(T) = \isc(T-R) + r + u_r$.  The upper bound actually does not depend on
$r+1$ being triangular.

\begin{lemma}\label{isctriub}
  $\isc(T) \leq \isc(T-R) + r + u_r$.
\end{lemma}
\begin{proof}
We give a strategy for Requester.  The strategy is similar to that in
in Lemma~\ref{lem:nontriang}, except that in Phase 2 we play a subgame on $T-R$
instead of $T-R-v$.  Again Requester first requests an initial color
$\alpha(x)$ for each $x \in R'$.  Again let $c_i$ be the $i$th color
supplied at $v$, and let $S_i=\{z \in R \st \alpha(z) = c_i\}$.

\medskip
{\it Phase 1}.  As long as $\C{S_i}>u_r-i+1$, Requester obtains color $c_{i+1}$
at $v$ and increments $i$.  When $\C{S_i}\le u_r-i+1$, no request is made,
Phase 2 begins, and we set $i^*=i$.  

{\it Phase 2}.  There are two cases, depending on whether $\C{S_{i^*}}$ equals
$u_r-{i^*}+1$ or is smaller.

If $\C{S_{i^*}}\le u_r-{i^*}$, then Requester proceeds as in
Lemma~\ref{lem:nontriang}.  Requester frees $c_{i^*}$ at $v$ and then
plays optimally on $T'$.  Again if Supplier adds $c_{i^*}$ to the list at $w$
in response to a request there, then Requester makes an extra request at $w$.
Without using the copy of $c_{i^*}$ in $L(w)$, Requester obtains lists on $T'$
from which a proper coloring can be chosen.  With $c_{i^*}$ available at $v$,
this coloring $\phi$ extends to $R'$.

If $\C{S_{i^*}}=u_r-{i^*}+1$, then Requester next plays an optimal game on
$T-R$, treating $L(v)$ as initially empty in the game on $T-R$.  The first time
Requester's optimal strategy on $T-R$ makes a request at $T-R$, we treat it in
the game on $T-R$ as being supplied by the color $c_{i}^*$ at $v$ that was
supplied earlier.  Hence we save one from $\isc(T-R)$ in counting the
requests.  From the resulting lists, choose a proper coloring $\phi$ of $T-R$,
using $c_{i^*}$ at $v$ if possible.  Since $v$ is a leaf in $T-R$, this is
possible unless $c_{i^*}$ must be used at $w$.  Let $c'=\phi(v)$.  If
$c'= c_{i^*}$, then Requester frees $c_{i^*}$ at $v$ and extends $\phi$ to all
of $T$.  If $c'\ne c_{i^*}$, then the process moves to Phase 3.

{\it Phase 3}.  Here, in the situation $\C{S_{i^*}}=u_r-i^*+1$ and
$c'\ne c_{i^*}$, Requester increments $i$ (to $i^*+1$) and requests $c_{i^*+1}$ at $v$.  As long as $\C{S_i}>u_r-i+1$, Requester increments $i$ and requests a
color at $v$.  At the point when $\C{S_i}\le u_r-i+1$, Requester frees
$c_i$ at $v$.  Since color $c_{i^*}$ already exists in $L(v)$ and has been
used at $\phi(w)$, the color $c_i$ is different from $c_{i^*}$ and can 
be used at $v$.  With $\phi(v)=c_i$, the coloring now extends to all of $T$.

\medskip

The same counting argument as in Lemma~\ref{lem:nontriang} shows that the 
strategy terminates, and we have argued that it produces lists from which
a proper coloring can be chosen.  We claim that it makes at most
$\isc(T-R)+r+u_r$ requests.

If $\C{S_{i^*}}\le u_r-i^*$, then as argued in Lemma~\ref{lem:nontriang} the
total number of requests is at most
  $ r + i^* + \sizeof{S_{i^*}} + (\isc(T') + 1) $.
Since $\C{S_{i^*}}\le u_r-i^*$, the number of requests is at most
$r+u_r+1+\isc(T')$.  Since $\isc(T-R)\ge\isc(T')+1$ by~\eqref{iscbd},
we obtain $\isc(T)\le r+u_r+\isc(T-R)$, as desired.

If $\C{S_{i^*}}= u_r-i^*+1$, then let $c_k$ be the last color supplied at $v$.
Now the total number of requests is at most
  $ r + k + \sizeof{S_k} + (\isc(T-R)-1), $
where as remarked earlier we save one request on $T-R$ because Requester uses
the already-counted first request made at $v$.  Since $\C{S_k}\le u_r-k+1$,
again the desired bound holds.
\end{proof}

Finally, we prove the lower bound in the triangular case.  To simplify
arguments for optimal strategies for Supplier, we prove a lemma with two
statements about a fixed vertex $v$ in the game on any graph $G$.  The first
property is that Supplier has an optimal strategy with the freedom to name the
first color supplied at $v$ independently of what Requester does.  This
property is a special case of Observation~2.1 of Bonamy and
Meeks~\cite{bonamy-meeks}.  We include a proof not only for completeness, but
also to show that this and the second property can be guaranteed simultaneously.

The second property restricts the strategies that need to be considered for
Requester in response, showing that it is non-optimal for Requester to make too
many requests at one vertex.  The goal of Requester on a graph $G$ is to
produce a list assignment $L$ such that $G$ is $L$-colorable.  Just as it is
useless in $f$-choosability to have $f(v)\ge d(v)+1$, where $d(v)$ denotes the
degree of $v$ in $G$, so it is nonoptimal for Requester to request more than
$d(v)+1$ colors at a vertex.  Such requests allow Supplier to increase the
final score.

Note that once a Requester strategy $\cR$ and a Supplier strategy $\cS$ on a
graph $G$ are fixed, the sequence of moves is fully determined.  We refer
to this sequence as the \emph{$(\cR, \cS)$-game}.

\begin{lemma}\label{lem:initialcolor} \label{lem:excess}
For a fixed vertex $v$ in a graph $G$, Supplier has an optimal strategy $\cS$
for the game on $G$ having the following two properties:\\
(1) $\cS$ always provides the same color $c_1$ at $v$ in response to the first
request made at $v$, regardless of what earlier requests have already been
made.\\
(2) For any Requester strategy $\cR$, if there are $d(v)+1+q$ requests at $v$
in the $(\cR,\cS)$-game, where $q\ge0$, then the total number of requests in
the $(\cR,\cS)$-game is at least $\isc(G)+q$.  Furthermore, all colors
supplied at $v$ by $\cS$ after the first $d(v)$ colors can be chosen
arbitrarily, as long as they have not yet been supplied at $v$.
\end{lemma}
\begin{proof}
Let $\cS'$ be an optimal Supplier strategy.  We define $\cS$ in terms of
$\cS'$.  First, let $c_1$ be the color that $\cS'$ would supply at $v$ if
Requester made the first request of the game at $v$.  Under $\cS$, Supplier
pretends that this move has been made against $\cS'$ and continues playing
according to $\cS'$ as if $c_1$ has already been supplied at $v$, until the
first actual request is made at $v$.  Strategy $\cS$ then actually supplies
$c_1$ at $v$, but pretends that no request was actually made then, since in
the imagined game played by $\cS'$ that request came earlier.  This point must
be reached, because the actual game cannot end without a request being made
at $v$.

Strategy $\cS$ then continues playing the game according to $\cS'$ (as long as
at most $d(v)$ requests have been made at $v$).  This can be done since the
lists $L(w)$ under the actual game played by $\cS$ and the imagined game played
by $\cS'$ are now the same for all $w\in V(G)$.

Strategy $\cS$ may also deviate from $\cS'$ when more than $d(v)$ requests are
made at $v$.  When Requester brings the list size at $v$ to $d(v)+1$, let 
$c'$ be the color that $\cS'$ would supply.  Strategy $\cS$ ignores this
and instead supplies any color $\hat c$ not already used at $v$, playing the
subsequent game as if $c'$ had actually been supplied.  In response to requests
beyond $d(v)+1$, $\cS$ supplies any colors not previously supplied at $v$
(``free colors''), and for other moves $\cS$ plays $\cS'$ as if the requests
beyond $d(v)+1$ requests at $v$ have not occurred.

The game continues until a proper coloring $\phi$ can be chosen from the lists.
When this occurs, colors supplied at $v$ after the first $d(v)+1$ requests can
be ignored, since $\phi$ can be chosen without using them (because at most
$d(v)$ colors occur on neighbors of $v$).  Also, if $\phi$ requires using
$\hat c$ on $v$, then the first $d(v)$ colors supplied by $\cS'$ must be used
by $\phi$ on the neighbors of $v$.  This means that in the imagined game played
according to $\cS'$, the color $c'$ at $v$ completes a proper coloring $\phi'$.
Since $\cS'$ is optimal, at least $\isc(G)$ requests must have been made in
addition to the $k-d(v)-1$ made beyond the first $d(v)+1$ requests at $v$.
\end{proof}

\begin{lemma}\label{isctrilb}
  If $r+1$ is triangular, then $\isc(T) \geq \isc(T-R) + r + u_r$.
\end{lemma}
\begin{proof}
The degree of $v$ in $T-R$ is $1$.  Lemma~\ref{lem:excess} allows us to fix an
Supplier strategy $\cS$ on $T-R$ such that, if Requester makes $2+q$ requests
at $v$, where $q\ge0$, then the total number of requests made on $T-R$ is at
least $\isc(T-R)+q$.  Furthermore, Lemma~\ref{lem:excess} guarantees that $\cS$
can fix the sequence of colors to be supplied at $v$ (in the game on $T-R$)
independently of what Requester does, so that the same colors are always
supplied at $v$ in the same order.  Let $\VEC c1{u_r+1}$ be the first $u_r+1$
colors planned by Supplier to be supplied to $v$, independently of the
Requester strategy $\cR$.

We extend $\cS$ to a strategy on all of $T$ by specifying an initial color
$\alpha(x)$ for each $x\in R$, to be supplied in response to the first request
at $x$, and thereafter supplying arbitrary colors at $x$ in response to
further requests.  To specify $\alpha$, note that
$r=-1+\SE i1{u_r+1}(u_r+2-i)$, since $r+1$ is triangular and
$u_r=\max\{k\st \CH{k+1}2\le r\}$.  Supplier gives initial color $c_1$ to
$u_r$ vertices in $R$.  For $2\le i\le u_r+1$, Supplier gives initial color
$c_i$ to $u_r-i+2$ vertices in $R$.


We claim that $\cS$ forces Requester to make at least $\isc(T-R)+r+u_r$
requests.  Let $\cR$ be an optimal strategy for Requester, and let $c$ be the
color given to $v$ in a proper coloring chosen from the lists when the game
ends.
  
If $c\notin\{\VEC c1{u_r+1}\}$, then Requester has made at least $u_r+2$
requests at $v$.  Hence Requester has made at least $\isc(T-R)+u_r$ requests
on $T-R$.  Together with the $r$ initial requests at $R$, at least
$r+u_r+\isc(T-R)$ requests have been made, as desired.

If $c = c_i$ for some $i$, then Requester has made at least $i$ requests at $v$
plus an extra request at each $x\in R$ such that $\alpha(x)=c_i$.
If $i=1$, then explicitly this counts at least $r+u_r+\isc(T-R)$ requests.
If $i\ge2$, then Requester has made at least $\isc(T-R)+i-2$ requests on 
$T-R$ and $r+u_r-i+2$ requests on $R$, again at least $r+u_r+\isc(T-R)$.
\end{proof}

Lemmas~\ref{lem:nontriang}, \ref{isctriub}, and \ref{isctrilb} complete the
proof of Theorem~\ref{thm:iscmain}.  We note that the arguments can be 
modified to handle also the case of stars.

\providecommand{\bysame}{\leavevmode\hbox to3em{\hrulefill}\thinspace}
\providecommand{\MR}{\relax\ifhmode\unskip\space\fi MR }
\providecommand{\MRhref}[2]{%
  \href{http://www.ams.org/mathscinet-getitem?mr=#1}{#2}
}
\providecommand{\href}[2]{#2}

\end{document}